\documentclass{article}
\usepackage[utf8]{inputenc}
\usepackage[a4paper]{geometry}
\usepackage{multirow}
\usepackage{array}
\usepackage[all]{xy}
\usepackage{tikz}

\usepackage{amsmath, amsfonts, amssymb, amsthm} 
\usepackage{biblatex}
\usepackage{mathrsfs}
\usepackage{mathtools}
\usepackage{enumitem} 
    \setlist[enumerate]{label=(\roman*),noitemsep,topsep=3pt} 
    \setlist[itemize]{label={$\cdot$},noitemsep,topsep=3pt} 
\usepackage{verbatim} 
\usepackage{url}
\usepackage{theoremref} 
\usepackage{bbm}
\usepackage{color}
\usepackage{float}
\definecolor{gray}{rgb}{0.5,0.5,0.5}
\usepackage{subcaption}
\usepackage{adjustbox}
\usepackage{tikz}
    \usetikzlibrary{matrix}
    \usetikzlibrary{backgrounds}
    \usetikzlibrary{positioning}
    
    \newdimen\nodeSize
    \newdimen\nodeDist
    \tikzset{
        position/.style args={#1:#2 from #3}{
            at=(#3.#1), anchor=#1+180, shift=(#1:#2)
        }
    }

\usepackage{listings}

\usepackage{rotating}
\usepackage{makecell}
\usepackage{tabu}
\usepackage{multirow} 
\usepackage{hyperref}
\hypersetup{colorlinks,allcolors=black}

\makeatletter
\newcommand\ackname{Acknowledgements}
\if@titlepage
  \newenvironment{acknowledgements}{%
      \titlepage
      \null\vfil
      \@beginparpenalty\@lowpenalty
      \begin{center}%
        \bfseries \ackname
        \@endparpenalty\@M
      \end{center}}%
     {\par\vfil\null\endtitlepage}
\else
  
\fi
\makeatother

\DeclareUnicodeCharacter{0301}{\'}
\DeclareUnicodeCharacter{1EF3}{y\'}

\newtheorem{theorem}{Theorem}[section]
\newtheorem{corollary}[theorem]{Corollary}
\newtheorem{proposition}[theorem]{Proposition}
\newtheorem{lemma}[theorem]{Lemma}

\theoremstyle{definition}

\newtheorem{remark}[theorem]{Remark}

\newtheorem{example}[theorem]{Example}

\newcommand{\thistime}{\expandafter\calctimeA\pdfcreationdate\@nil} 
\def\calctimeA#1:#2#3#4#5#6#7#8#9{\calctimeB}
\def\calctimeB#1#2#3#4#5\@nil{#1#2:#3#4}

\newcommand{\note}[1]{\textcolor{red}{#1}}
\newcommand{\fs}[1]{\textcolor{cyan}{#1}}

\newcommand{\zen}{\mathsf{Z}} 

\renewcommand{\S}{\mathsf{S}}
\newcommand{\A}{\mathsf{A}}

\newcommand{\D}{\mathsf{D}}

\newcommand{\numberset}[1]{\mathbbm{#1}}
\newcommand{\N}{\numberset{N}}
\newcommand{\Z}{\numberset{Z}}

\newcommand{\sym}{\operatorname{\mathsf{Sym}}}

\newcommand{\aut}{\operatorname{\mathsf{Aut}}}
\newcommand{\lmlt}{\operatorname{\mathsf{LMlt}}}
\newcommand{\dis}{\operatorname{\mathsf{Dis}}}
\newcommand{\aff}{\operatorname{\mathsf{Aff}}}

\renewcommand{\ker}{\operatorname{\mathsf{ker}}}

\newcommand{\im}{\operatorname{\mathsf{Im}}}

\renewcommand{\min}{\operatorname{\mathsf{min}}}

\newcommand{\conj}{\operatorname{\mathsf{Conj}}}
\newcommand{\core}{\operatorname{\mathsf{Core}}}

\renewcommand{\O}{\operatorname{\mathcal{O}}}

\newcommand{\fix}{\operatorname{\mathsf{Fix}}}
\newcommand{\+}{\mkern2mu} 
\let\perogni\forall 
\let\esiste\exists
\renewcommand{\forall}{\perogni\+}
\renewcommand{\exists}{\esiste\+}

\def\setof#1#2{\{#1\,\colon #2\}}

\newcommand{\mo}{^{-1}} 

\renewcommand{\t}{\rhd}

\title{On Core Quandles}
\author{Filippo Spaggiari\footnote{Partially supported by GAUK and the GAČR grant no. 22-19073S} and Marco Bonatto\\ 
\small{\href{spaggiari@karlin.mff.cuni.cz}{spaggiari@karlin.mff.cuni.cz}}\\\small{Department of Algebra, Faculty of Mathematics \& Physics, Charles University Prague, Czech Republic}}

\date{\today}

\begin{document}

\maketitle

\begin{abstract}
    We characterize several properties of core quandles in terms of the properties of their underlying groups. Specifically, we characterize connected cores providing an answer to an open question in \cite{saito} and present a standard homogeneous representation for them, which allows us to prove that simple core quandles are primitive.
\end{abstract}


\section*{Introduction}

Quandles are binary algebras introduced in the development of the theory of knots as a way to construct knot invariants \cite{J, Matveev}. The origins of quandle theory trace back to self-distributive quasigroups, or what are now referred to as latin quandles, which have been extensively studied in relation to modern theory. 

Later on, they have proven to be deeply connected to various other mathematical structures, such as solutions to the set-theoretical quantum Yang-Baxter equation \cite{etingof2001indecomposable, etingof1999set} and pointed Hopf algebras \cite{AG}. Historical examples also stem from group conjugation, which ultimately led to the development of the knot quandle by Joyce and Matveev. Another line of research was inspired by the abstract properties of reflections on differentiable manifolds, giving rise to what is now known as involutory quandles, that is the class of structures where the identity $x * (x * y) = y$ holds \cite{Loos}. 

Core quandles represent a particular class of involutory quandles constructed from a group $ G $, defined by the operation $ x * y = xy\mo x $ for all $ x, y \in G $. The study of cores as binary algebraic structures has its roots in the analysis of quasigroups and loops (see \cite{belousov1967foundations}), which served as a foundation for the development of modern core theory.

From the perspective of universal algebra, core quandles can be regarded as reducts of groups, inheriting numerous structural properties from the operation of $ G $. Indeed, several properties of the quandle $ \core(G) $ can be derived directly from those of the underlying group $ G $. The paper was motivated by an open question in \cite{saito}, in which a characterization of connected quandles was asked. Previous investigations have addressed properties such as latinness \cite{kano1976finite}, faithfulness \cite{umaya1976symmetric}, right self-distributivity \cite{belousov1963class}, and mediality (an unpublished result by B. Roszkowska). This paper contributes to this ongoing research on core quandles in purely group-theoretic terms by providing a characterization of other properties, among which connectedness, simplicity, primitivity, abelianness and nilpotency in the sense of \cite{comm}. To ensure comprehensiveness, we also revisit existing results, rephrasing them in the modern framework of quandle theory and presenting independent, complete proofs.

The structure of the paper is as follows. In Section \ref{sec: preliminaries}, we introduce the foundational concepts and notations of quandle theory. 

Section \ref{sec: hom rep} examines the homogeneous representations of core quandles and the properties of their associated groups of automorphisms. In particular, we use these homogeneous representations to characterize simple and primitive core quandles. In Section \ref{sec dis}, we focus on the displacement group and we establish the characterization of connected core quandles.

In Section \ref{properties}, we investigate the above mentioned characterization theorems for several properties of core quandles in purely group theoretical terms.

In Sections \ref{sec: commutator theory}, we employ tools from universal algebra to examine the notions of abelianness and nilpotence in core quandles (in the sense of \cite{comm}).

In Section \ref{sec: functors} we address the problem of embedding into core quandles, also investigated in \cite{bergman2021core}. Using a categorical language we obtained a characterization of quandles that can be embedded into core quandles. 
Finally, in Section \ref{sec: coloring} we focus on knot colorings by core quandles.

The research presented in this work has been facilitated by the computational tools RIG \cite{RIG} and GAP \cite{GAP4}.

\subsection*{Notation}
Let us collect the group-theoretical notation we are using in the paper. Let $G$ be a group. We denote by $G^{op}$ the \emph{opposite} of $G$, that is the group with operation $x\cdot^{op} y=yx$ for every $x,y\in G$. The map $\iota:x\mapsto x^{-1}$ is an isomorphism between $G$ and $G^{op}$. Let $G$ be acting on a set $X$. We denote by $G_x$ the \emph{stabilizer} in $G$ of $x\in X$, and by $x^G$ the orbit of $x$ under the action of $G$. We say that (the action of) $G$ is {\it semiregular} if $G_x=1$ for every $x\in G$ and {\it regular} if is it transitive and semiregular. We will employ the following standard notation for the canonical left and right actions of a group $G$ on itself: 
\begin{eqnarray*}
\lambda: G\to \sym(G),&\quad  &g\mapsto (\lambda_g: x\mapsto gx),\\ 
\rho: G^{op}\to \sym(G),&\quad  &g\mapsto (\rho_g:x\mapsto xg). 
\end{eqnarray*}
Note that both $\lambda$ and $\rho$ are regular actions and that $[\lambda_g,\rho_h]=1$ for every $g,h\in G$. We also denote by $\widehat{g}$ the inner automorphism by $G$, mapping $x$ to $\lambda_g\rho_g^{-1}(x)=gxg^{-1}$ for every $x\in G$.
We define the {\it squaring map} as $s:G\to G$ such that $s(x)=x^2$ for every $x\in G$. The image $s(G)=\{x^2\colon x\in G\}$ is the \emph{square} of the group $G$.

\section{Core quandles} \label{sec: preliminaries}

Let $Q=(Q, *)$ be a set with a binary operation. We say that $Q$ is a \emph{rack} if all left translations
$$
L_x: Q \rightarrow Q, \quad y \mapsto x*y
$$
are automorphisms of $Q$. Equivalently, we have that $L_x$ is bijective and
$$x*(y*z)=(x*y)*(x*z)$$
holds for every $x,y,z\in Q$. 
We have introduced quandles as algebras with a single binary operation, but since $L_x\mo L_x (y)=y$, it is also common to read $L_x\mo (y)$ as a second binary operation, the left division of the quandle, denoted $x\backslash y$. In the finite case the algebraic structures $(Q,*)$ and $(Q,*,\backslash)$ are equivalent, so we can study finite racks avoiding the division without loss of generality. If a rack $Q$ is idempotent, that is, if $x*x=x$ for all $x \in Q$, then $Q$ is called \emph{quandle}. In particular we can define the {\it left multiplication group} of $Q$ as
$$
\lmlt(Q)=\left\langle L_x: x \in Q\right\rangle
$$
and the \emph{displacement group} of $Q$ as
$$
\dis(Q)=\left\langle L_xL_y\mo : x,y \in Q\right\rangle.
$$
It is known that $\lmlt(Q)$ is a normal subgroup of the automorphism group $\aut(Q)$. 

A rack $Q$ is said to be:
\begin{enumerate}
    \item \emph{homogeneous} (resp. \emph{connected}) if the natural action of $\aut(Q)$ (resp. $\lmlt(Q)$) is transitive on $Q$.
    \item \emph{projection} if $x*y=y$ for all $x,y\in Q$.
    \item a \emph{crossed set} if it is a quandle and for all $x,y\in Q$ we have $x\ast y = y$ if and only if $y\ast x = x$.
\end{enumerate}

Quandles can be constructed using groups and their automorphisms in several ways.

\begin{example}
\text{ }
\begin{enumerate}
\item Let $G$ be a group. It is easy to see that $G$ endowed with the binary operation
\begin{align*}
    x* y &= xyx\mo 
\end{align*}
is a quandle, called \emph{conjugation quandle} over $G$ and denoted $\conj(G)$.

\item Let $G$ be a group, $f\in \aut(G)$ and $H\leq \fix(f)=\setof{x\in G}{f(x)=x}$. The set $G/H$ endowed with the operation 
$$xH*yH=xf(x\mo y)H$$ 
for every $x,y\in G$ is a quandle. We denote such quandle as $\mathcal{Q}(G,H,f)$ and we refer to it as {\it coset quandle}. If $G$ is abelian and $H=1$, we say that it is {\it affine}. In this case, we use the notation $\aff(G,f)$.
\end{enumerate}
\end{example}

It can be proven that homogeneous quandles are exactly those that can be built using the above coset construction (see, for instance \cite{hsv}). Note that the statement of \cite[Proposition 3.5]{hsv} holds even under slightly weaker assumptions, by using the same proof.

\begin{proposition}[\cite{hsv}, Proposition 3.5]\label{homogeneous_rep}
    Let $Q$ be a quandle and let $x\in Q$. Let $G\leq \aut(Q)$ be acting transitively on $Q$ and such that $\widehat{L_x}(G)=G$. Then $Q\cong \mathcal{Q}(G,G_x,\widehat{L_x})$.
\end{proposition}

The displacement group plays a very important role for quandles, for instance it can be used to build connected quandles with the coset representation (indeed the orbits of the displacement group and the left multiplication group coincide). 

\begin{proposition}[\cite{hsv}, Theorem 4.1, and \cite{J}]
    Let $Q$ be a connected quandle and $x\in Q$. Then $Q\cong \mathcal{Q}(\dis(Q),\dis(Q)_x,\widehat{L_x})$.
\end{proposition} 
Let $G$ be a group. The binary operation
    \begin{align*}
        x* y &= xy\mo x 
    \end{align*}
endows $G$ with a quandle structure called \emph{core} of $G$, and denoted by $\core(G)$. It is easy to prove that $\core(G)$ is an \emph{involutory} quandle, that is $L_x^2=1$ for all $x\in Q$.

\begin{example}\text{}
\begin{enumerate}
\item Let $A$ be an abelian group. Then $\core(A)=\aff(A,-1)$.
\item    It is easy to prove that $\core(G)$ is projection if and only if $G$ is elementary $2$-abelian.

\end{enumerate}
\end{example}

Core quandles are {\it reducts} of groups, so we have the following.

\begin{remark} \label{not iso iff}
\text{}
    \begin{enumerate}
    \item Every homomorphism of groups $G\to H$ gives rise to a homomorphism of core quandles $\core(G)\to\core(H)$. The converse does not hold, indeed $\core(\Z_3^2\rtimes \Z_3)$ and $\core(\Z_3^3)$ are isomorphic as quandles (but their undelying groups are not). 
    
        \item If $G$ is a group and $H$ is a normal subgroup of $G$ then the canonical map $G\longrightarrow G/H$ is a onto quandle homomorphism from $\core(G)$ onto $\core(G/H)$.  
               
    \end{enumerate}
\end{remark}

We start by analyzing the closure properties of the class of core quandles from a universal algebraic point of view.

\begin{proposition}
    The class of core quandles is closed under direct products.
\end{proposition}

\begin{proof}
    It is easy to check that $\prod_{i\in I} \core(G_i)=\core (\prod_{i\in I} G_i)$ for an arbitrary set of groups $\setof{G_i}{i\in I}$.
    %
\end{proof}

\begin{example}
\text{}
\begin{enumerate}
\item    The class of core quandles is not closed under subalgebras. In fact, the quandle $\core(\A_4)$ contains a copy of a projection quandle of size $3$, which is not the core quandle of $\Z_3$ (the only group of size 3).

  \item  The class of core quandles is not closed under quotients (hence, under homomorphic images). Indeed, the core quandle $Q=\core(\A_4)$ has a congruence $\theta$ such that $Q/\theta$ is the connected quandle {\tt SmallQuandle(6,1)} of the RIG database.
%
%
    None of the quandles $\core(G)$ is isomorphic to $Q/\theta$, for $G$ running over all the groups of order $6$.
    \end{enumerate}
    \end{example}
\section{Homogeneous representations} \label{sec: hom rep}
These canonical actions of $G$ on itself and the action by conjugation of $G$ on $G$ are actions by automorphisms of $\core(G)$. Indeed 
     \begin{align*}
    xy*xz=xyz^{-1} x^{-1} xz=x(yz^{-1}y)=x(y*z),\\
    yx*zx=yxx^{-1}z^{-1} yx=(yz^{-1}y)x=(y*z)x,
     \end{align*}
     for every $x,y,z\in Q$.
Therefore core quandles are homogeneous. In this section we provide a standard homogeneous representation obtained by using the canonical left and right actions of the underlying group.

\begin{lemma}
Let $Q=\core(G)$ be a quandle, and let $K=\setof{(x,x\mo )}{x\in \zen(G)}$. Then
$$\varphi:G\times G^{op}\to \aut{(Q)}\quad (x,y)\mapsto \lambda_x\rho_y$$
is a group homomorphism and $\ker{\varphi}=K\cong\zen(G)$. Consequently, $(G\times G^{op})/K\cong\lambda(G)\rho(G)$.
\end{lemma}

\begin{proof}
Let $x,y,z,t\in G$. Then
$$\varphi(x,y)\varphi(z,t)=\lambda_x\rho_y\lambda_z\rho_t=\lambda_{x}\lambda_{z}\rho_y\rho_t=\lambda_{xz}\rho_{ty}=\varphi(xz,ty)=\varphi((x,y)(z,t)).$$
If $(x,y)\in \ker{\varphi}$ then $\lambda_x \rho_y(z)=xzy=z$ for every $z\in G$. If $z=1$ we have that $y=x\mo $. Then $xzx\mo =z$ for every $z\in G$, thus $x\in \zen(G)$.
\end{proof}

We can actually provide two different homogeneous representations that will be used later on. Let $G$ be a group, $t\in\mathbb{N}$, and $\theta\in \aut(G)$. We denote by $\theta_t$ the map 
$$G^t\to G^t,\quad (x_1,\ldots,x_t)\mapsto (\theta(x_t),x_1,\ldots, x_{t-1})$$
for every $x_1,\ldots, x_t\in G$. Note that $\theta_t$ is an automorphism of $G^t$. 

\begin{proposition}\label{core homo}
    Let $Q=\core(G)$. Then:
    \begin{itemize}

    \item[(i)] $Q\cong \mathcal{Q}(G\times G^{op},H,\hat{\iota})$, where $H=\setof{(x,x\mo )}{x\in G}$.
    \item[(ii)] $Q\cong \mathcal{Q}(G^2 ,D,1_2)$, where $D=\setof{(x,x)}{x\in G}$ and $1_2(x,y)=(y,x)$ for every $x,y\in G$.
    \end{itemize}
\end{proposition}

\begin{proof}
(i)    Identify the group $G\times G^{op}$ with $\rho(G)\lambda(G)$. The group $G\times G^{op}$ is transitive on $Q$ and it is $\hat{\iota}$-invariant, since for all $x,y\in G$ we have $\iota\lambda_x\rho_y\iota=\rho_{x\mo}\lambda_{y\mo}=\lambda_{y\mo}\rho_{x\mo}$. Thus, according to Proposition \ref{homogeneous_rep} we have that $Q\cong \mathcal{Q}(G\times G^{op},H, \hat{\iota})$, where $H$ is the stabilizer of $1$. In particular we have that $(x,y)\in (G\times G^{op})_1$ if and  only if $xy=1$, that is $H=\setof{(x,x\mo )}{x\in G}$. 


(ii) The map $(x,y)\mapsto (x,y\mo )$ from $G\times G^{op}$ to $s(G)$ provides and isomorphism of quandles between $\mathcal{Q}(G\times G^{op},H,\iota)$ and $\mathcal{Q}(G^2,D,1_2)$.
\end{proof}

The first consequence of the homogeneous representation we obtained above is the following characterization of core quandles among involutory quandles in general.

\begin{proposition}\label{characterization of core}
    Let $Q$ be an involutory quandle and $x_0\in Q$. The following conditions are equivalent.
    \begin{enumerate}
        \item $Q\cong \core(G)$.
        \item There exist regular actions $\lambda\colon G\to\aut(Q)$ and $\rho\colon G^{op}\to \aut(Q)$  such that $[\lambda_g,\rho_h]=1$, $L_{x_0} \lambda_g L_{x_0}=\rho_g\mo $ and $\lambda_g(x_0)=\rho_h(x_0)$ implies $h=g$ for all $g,h\in G$.
    \end{enumerate}
\end{proposition}

\begin{proof}
   (i) $\Rightarrow$ (ii) If $Q\cong\core(G)$, it is enough to take the canonical left and right actions and $x_0=1$.

   (ii) $\Rightarrow$ (i) Conversely, define $Q'=\mathcal{Q}(\lambda(G)\times \rho(G^{op}),H,f)$ where $H=\setof{(\lambda_g,\rho_g^{-1})}{g\in G }$ and $f(\lambda_g,\rho_h)=L_{x_0}(\lambda_g,\rho_h)L_{x_0}=(\lambda_h\mo ,\rho_g\mo )$ for every $g,h\in G$. The map $(\lambda_g,\rho_h)H\mapsto \lambda_g \rho_h(x_0)$ is an isomorphism of quandle. On the other hand the map, let $K=\setof{(g,g\mo )}{g\in G}\leq G\times G^{op}$ and define the map $(\lambda_g,\rho_h)H\mapsto (g,h)K\in (G\times G^{op})/K$. Such map provides and isomorphism between $Q'$ and $Q''=\mathcal{Q}(G\times G^{op},K,f)$ where $f(g, h)=(h\mo ,g\mo )$ for every $g,h\in G$. According to Proposition \ref{core homo} $Q''$ is isomorphic to $\core(G)$. 
\end{proof}

We can also use the homogeneous representation obtained in Proposition \ref{core homo} to characterize simple and primitive core quandles. Recall that a quandle is \emph{simple} if the congruence lattice of $Q$ has only two elements, and it is \emph{primitive} if the natural action of $\lmlt(Q)$ on $Q$ is primitive.

In \cite{Primitive} we investigated primitive and simple quandles and we denoted the quandle $\mathcal{Q}(G^t,Fix(\theta_t),\theta_t)$ just by $(G,t,\theta)$. For our purposes note that the quandle $(G,2,1)$ is just $\mathcal{Q}(G^2,D,1_2)$ where $D=\setof{(x,x)}{x\in G}$ as defined in Proposition \ref{core homo}. 

\begin{theorem}
Let $Q=\core(G)$. The following are equivalent.

\begin{itemize}
\item[(i)] $Q$ is primitive.

\item[(ii)] $Q$ is simple.

\item[(iii)] $G$ is simple.
\end{itemize}
\end{theorem}

\begin{proof}
(i) $\Rightarrow$ (ii) This implication is true in general.

(ii) $\Rightarrow$ (iii) Every normal subgroup of $G$ would give rise to a quandle congruence of $Q$ (see Remark \ref{not iso iff}), so if $\core(G)$ is simple, then also $G$ must be simple. 

(iii) $\Rightarrow$ (i) Let $G$ be simple. Then $\core(G)\cong \mathcal{Q}(G^2,D,1_2)$ where $D=\setof{(x,x)}{x\in G}$ (see Proposition \ref{core homo}). Using the notation introduced in \cite{Primitive} then $\core(G)\cong (G,2,1)$, so according to \cite[Theorem 4.10]{Primitive} $Q$ is primitive.
\end{proof}






\section{The displacement group of core quandles}\label{sec dis}
Let $Q=\core(G)$. Let us first focus on the action of $\lmlt(Q)$ and $\dis(Q)$. Note that $L_1(x)=\iota(x)=x\mo $ and that $L_x L_1(y)=xyx$ for every $x,y\in G$. We denote the map $L_x L_1$ by $\delta_x$ for every $x\in G$.

It is known that the displacement group of a quandle $Q$ is generated by $\setof{L_x L_{x_0}\mo }{x\in Q}$, therefore we can write $\lmlt(Q)=\dis(Q)\langle L_{x_0}\rangle$ for every $x_0\in Q$ \cite[Lemma 1.4]{semimedial}. With this property, if $Q=\core(Q)$ we have that
%
%
   \begin{align*}
   \dis(Q)&=\langle \delta_x\,\colon x\in Q\rangle,\\
   \lmlt(Q)&=\left\langle\iota,\delta_z\, \colon\, z\in G\right\rangle.
   \end{align*}


\begin{lemma} \label{dis_stab_1}
    Let $Q=\core(G)$. The following are equivalent 
    \begin{itemize}
\item[(i)]  $\delta_{x_1}\cdots\delta_{x_t}\in\dis(Q)_1$.
\item[(ii)] $x_1\cdots x_t=(x_t\cdots x_1)\mo$.
    \end{itemize}
    In particular, $\dis(Q)_1=\dis(Q)\cap \lmlt(G)$.
\end{lemma}

\begin{proof}
Note that $\delta_{x_1}\cdots\delta_{x_t}(1)=
x_1\cdots x_t x_t \cdots x_1=1$ if and only if $x_t\cdots x_1=(x_1\cdots x_1)^{-1}$. Thus (i) and (ii) are equivalent.

So, if $\delta_{x_1}\cdots\delta_{x_t}\in \dis(Q)_1$ then $\delta_{x_1}\cdots\delta_{x_t}\in \lmlt(G)$. On the other hand, $\lmlt(G)_1\leq \aut{Q}_1$. Therefore $\dis(Q)_1=\dis(Q)\cap \lmlt(G)$.
\end{proof}

\begin{proposition}
    Let $Q=\core(G)$ be a non-projection quandle. Then $\lmlt(Q)\cong\dis(Q)\rtimes\Z_2$.
\end{proposition}

\begin{proof}
    Let $I=\langle\iota\rangle\cong\Z_2$. We know that $\dis(Q)$ is normal in $\lmlt(Q)$ and that $\dis(Q)I=\lmlt(Q)$. Assume that $I\cap \dis(Q)\neq 1$, so we must necessarily have that $I\subseteq \dis(Q)$. In particular $\iota\in \dis(Q)_1$, and so according to Lemma \ref{dis_stab_1}, the map $\iota:x\mapsto x^{-1}$ is an inner automorphism of $G$. Therefore $G$ is abelian. Hence, $\lmlt(G)=1$ and so $\iota(x)=x^{-1}=x$ for every $x\in G$. Therefore $G$ is a $2$-elementary abelian group and $Q$ is projection. 
    %
%

    Thus, if $Q$ is not projection then $\dis(Q)\cap I=1$ and so $\lmlt(Q)=\dis(Q)\rtimes I$. 
\end{proof}

Let us now investigate the relationship between the group $G$ and the displacement group of the corresponding core quandle.

\begin{lemma}\label{lemma on action of dis}
Let $Q=\core(G)$. The following are equivalent 
\begin{itemize}
\item[(i)]  $\delta_{x_1}\cdots \delta_{x_t}=1$. 
\item[(ii)]  $x_1\cdots x_t=(x_t\cdots x_1)\mo \in \zen(G)$.
\end{itemize}
\end{lemma}

\begin{proof}
(ii) $\Rightarrow$ (i) If $x_1\cdots x_t=(x_t\cdots x_1)\mo \in \zen(G)$ then
\begin{align*}
    \delta_{x_1}\cdots \delta_{x_t}(a)=x_1\cdots x_t a x_t\cdots x_1=x_1\cdots x_t a (x_1\cdots x_t)\mo =a
\end{align*}
for every $a\in G$. Thus, $\delta_{x_1}\cdots \delta_{x_t}=1$.

(i) $\Rightarrow$ (ii) Assume that $\delta_{x_1}\cdots \delta_{x_t}=1$. Then
\begin{align*}
    \delta_{x_1}\cdots \delta_{x_t}=x_1\cdots x_t x_t\cdots x_1=1
\end{align*}
and so $x_1\cdots x_t=(x_t\cdots x_1)\mo $. Then 
\begin{align*}
    \delta_{x_1}\cdots \delta_{x_t}(a)=x_1\cdots x_t a x_t\cdots x_1=x_1\cdots x_t a (x_1\cdots x_t)\mo =a
\end{align*}
for every $a\in G$. Therefore $x_1\cdots x_t\in \zen(G)$.
\end{proof}

\begin{lemma}
Let $Q=\core(G)$ be a quandle, and let $K=\setof{(x,x\mo )}{x\in \zen(G)}$. Then
$$\psi:\dis(Q)\to \left(G\times G^{op}\right)/K,\quad \delta_{x_1}\cdots \delta_{x_t}\mapsto (x_1\cdots x_t,x_t\cdots x_1)K$$
is an injective group homomorphism.
\end{lemma}

\begin{proof}
    The map $\psi$ is a group homomorphism, indeed, on the generators of $\dis(Q)$ we have
    \begin{align*}
        \psi( \delta_{x_1}\delta_{x_2})=(x_1 x_2,x_2 x_1)K=(x_1,x_1)K(x_2,x_2)K=\psi(\delta_{x_1})\psi(\delta_{x_2}).
    \end{align*}

Let $g,h\in \dis(Q)$. According to Lemma \ref{lemma on action of dis}, $\psi(g)=\psi(h)$ if and only if $g=h$. Then $\psi$ is a well defined injective group homomorphism.
    \end{proof}

\begin{corollary}\label{G/Z(G) factor of dis}
Let $Q=\core(G)$. Then $G/\zen(G)$ is a quotient of $\dis(Q)$.
\end{corollary}

\begin{proof}
The map 
\begin{align*}
\dis(Q)\to G\times G^{op}/K\to \left(G/\zen(G)\right)^2\to G/\zen(G)
\end{align*}
mapping
\begin{align*}
\delta_{x_1}\cdots \delta_{x_t}\mapsto (x_1\cdots x_t,x_t\cdots x_1)K\mapsto (x_1\cdots x_t \zen(G),(x_t\cdots x_1)\mo \zen(G))
\mapsto x_1\cdots x_t \zen(G)
\end{align*}
is a surjective group homomorphism.
\end{proof}


Note that the map defined in the proof of Corollary \ref{G/Z(G) factor of dis} is not injective in general. Indeed if $G$ is abelian but not elementary $2$-abelian then $\dis(Q)\neq 1$ and $G/\zen(G)=1$.

We can identify the orbits with respect to the left multiplication group using the group generated by the image of the squaring map $s$, as follows.

\begin{proposition}\label{core orbits}
    Let $Q=\core(G)$. Then $x^{\lmlt(Q)}=x\langle s(G)\rangle$ for every $x\in Q$.
\end{proposition}

\begin{proof}
    Let $y\in Q$. Then $L_y(x)=y*x=yx\mo y=x(x\mo y)^2$. Let assume by induction that $L_{y_n}\cdots L_{y_1}(x)=xh$ for some $h\in \langle s(G)\rangle$. Then
    \begin{align*}
        L_{y_{n+1}}\cdots L_{y_1}(x)&=y_{n+1} (xh)\mo  y_{n+1}\\&
        =y_{n+1} h\mo x\mo  y_{n+1}\\
        &=x x\mo y_{n+1} h\mo  x\mo  y_{n+1}
        \end{align*}
        Since $\langle s(G)\rangle$ is normal in $G$ we have that $x\mo y_{n+1} h\mo =h'x\mo y_{n+1}$ for some $h'\in \langle s(G)\rangle$. So
        \begin{align*}
        L_{y_{n+1}}\cdots L_{y_1}(x)&=x  h' (x\mo  y_{n+1})^2\in x\langle s(G)\rangle.\qedhere
    \end{align*}
\end{proof}

The description of the orbits of the quandle allows us to characterize connectedness for cores, answering to the open question posted in \cite{saito}

\begin{corollary}\label{connected}
Let $G$ be a group and $Q=\core(G)$. The following are equivalent:
\begin{itemize}
\item[(i)] $Q$ is connected. 
\item[(ii)] $G=\langle s(G)\rangle$. 
\item[(iii)] $G=\setof{x_1\cdots x_tx_t\cdots x_1}{x_i\in G,\, t\in \mathbb{N}}.$
\end{itemize}
\end{corollary}

\begin{proof}

The equivalence between (i) and (ii) follows by Proposition \ref{core orbits}.

(i) $\Leftrightarrow$ (iii)    The orbit of $1$ under $\dis(Q)$ is $L=\setof{x_1\cdots x_tx_t\cdots x_1}{x_i\in G,\, t\in \mathbb{N}}$. Then $Q$ is connected if and only if $G=L$.
\end{proof}

\begin{example}
    For every integer $n\ge 3$, the core quandle $\core(\S_n)$ is not connected and it has two orbits: $\A_n$ and $\S_n\setminus\A_n$. 
    Indeed, it is known that an element $\sigma\in\S_n$ is a square if and only if for every even $k\in\N$ the cycle structure of $\sigma$ has an even number of $k-$cycles, therefore $\langle s(\S_n)\rangle=\A_n$ for every $n\ge 3$.
    Now, according to Proposition \ref{core orbits}, for every $x\in\S_n$ we have 
    \[
        \O(x)=x\langle s(\S_n)\rangle=x\A_n=
        \begin{cases}
            \A_n,\quad&\text{if } x\in\A_n \\
            \S_n,\setminus\A_n\quad&\text{if } x\in\S_n\setminus\A_n 
        \end{cases}
    \]
\end{example}


\section{Group theoretical characterizations of core quandle properties}\label{properties}

We are now prepared to present the characterization theorems concerning the properties of cores in terms of those of the underlying groups. We start by restating a result first seen in \cite{umaya1976symmetric}. The \emph{Cayley kernel} of a quandle $Q$ is the congruence relation $\lambda_Q$ defined by setting \(x \,\lambda_Q\, y\) if and only if \(L_x = L_y\). Recall that a quandle $Q$ is \emph{faithful} if and only if $\lambda_Q$ is the identity relation. 



\begin{lemma}
Let $G$ be a group, $Q=\core(G)$ and $H=\setof{x\in \zen(G)}{x^2=1}$. Then:
\begin{itemize}
\item[(i)] $L_x=L_y$ if and only if $x\mo y\in H$. \item[(ii)] $Q/\lambda_Q=\core(G/H)$.
\item[(iii)] $Q$ is faithful if and only if $\zen(G)$ contains no involutions.
\end{itemize}
\end{lemma}

\begin{proof}
(i) Let $y=xa$ for $a\in H$ then  
$$L_y(z)=xa z\mo xa=xz\mo xa^2=x z\mo x=L_x(z).$$

Conversely, if $L_x=L_y$ then $L_x(1)=x^2=L_y(1)=y^2$, that is $x\mo y=xy\mo $. Since for all $z\in G$ $L_x(z )=xz^{-1}x=yz^{-1}y=L_y(z)$. Using the equality above we have
$$z^{-1}=y\mo x z^{-1}x y\mo=y\mo x  z^{-1}x\mo y=(y\mo x)  z^{-1}(y\mo x)\mo,$$
so $y\mo x\in \zen(G)$. Since $y\mo x\in \zen(G)$, we have
$$(y\mo x)^2 = y\mo x y\mo x = y\mo y\mo xx = 1.$$

(ii) Using the statement above, it is enough to observe that
\[
     x\, \lambda_Q\, y \iff  x\mo y\in H .
\]
Thus, the kernel of the canonical map $\core(G)\to \core(G/H)$ is $\lambda_Q$.

(iii) Note that $\lambda_Q$ is the identity relation if and only if $H=1$, that is there are no involutions in $\zen(G)$.
%
\end{proof}

\begin{corollary}
Let $G$ be a finite group. Then $\core(G)$ is faithful if and only if $|\zen(G)|$ is odd. 
\end{corollary}

We proceed with a characterization of superfaithful core quandles. Recall that a quandle $Q$ is \emph{superfaithful} if every subquandle is faithful, or equivalently, if $\fix(L_x)=\{x\}$ for every $x\in Q$ \cite[Lemma 1.9]{Super}.

\begin{lemma}
    Let $Q=\core(G)$ and $x,y\in Q$. Then 
    
    \begin{itemize}
    \item[(i)] $x*y=y$ if and only if $|x\mo y|=2$. 
    \item[(ii)] $Q$ is a crossed set. 
    \item[(ii)]  $Q$ is superfaithful if and only if $G$ has no involutions.
    \end{itemize}

\end{lemma}

\begin{proof}
(i) Note that $x * y=xy\mo x=y$ if and only if $x\mo y=y\mo x=(x\mo y)\mo $, i.e. $|x\mo y|=2$.

(ii) Note that $|y\mo x|=|(y\mo x)\mo |=|x\mo y|$. Thus $x*y=y$ if and only if $y*x=x$. Thus, $\core(G)$ is a crossed set.

(iii) Let $x,y\in Q$ be such that $y\in Fix(L_x)$, i.e. $x\ast y = y$. Thus, according to (i) $|x\mo  y|= 2$. So, $Q$ is superfaithful if and only if $G$ has no involutions.
\end{proof}





A quandle $Q$ is \emph{latin} if also all right translations (defined by $R_x:y\mapsto y*x$ for every $x,y\in Q$) are bijective maps. Recall that a group $G$ is said to be \emph{uniquely 2-divisible} if the squaring map $s:x\mapsto x^2$ is bijective. 

\begin{proposition}[\cite{kano1976finite}]
    $\core(G)$ is latin if and only if $G$ is uniquely 2-divisible.
\end{proposition}

\begin{proof}
    The quandle $\core(G)$ is homogeneous, so $Q$ is latin if and only if $R_1$ is bijective.
    Note that for a core quandle $Q=\core(G)$, we have $R_1(x)=x*1=x^2=s(x)$. So, $Q$ is latin if and only if it is uniquely 2-divisible.
\end{proof}


A finite involutory quandle $Q$ is latin if and only if $Q$ is superfaithful \cite{Super}. So we have the following.

\begin{corollary} \label{superf_cuperc_latin}
    Let $Q=\core(G)$ be finite. The following conditions are equivalent.
        \begin{enumerate}
            \item $Q$ is superfaithful;
            \item $Q$ is latin;
            \item $|G|$ is odd.
        \end{enumerate}

\end{corollary}

    Note that the equivalence of Corollary \ref{superf_cuperc_latin} does not hold for infinite core quandles. Consider $\core(\mathbb{Z})$ as a counterexample.

\begin{proposition}
Let $G$ be a finite nilpotent group and $Q=\core(G)$. Then $Q$ is connected if and only if $Q$ is latin. 
\end{proposition}

\begin{proof}
Latin quandles are connected. Assume that $Q$ is connected and $K$ be the $2$-Sylow of $G$. The group $G$ is the direct product of its Sylow subgroups. So we have that the canonical projection onto $K$ is a quandle morphism from $\core(G)$ and $\core(K)$. Let $\Phi(K)$ be the Frattini subgroup of $K$. The quandle $\core(K/\Phi(K))$ is a factor of $\core(K)$. Since $K/\Phi(K)$ is a $2$-elementary abelian group the $\core(K/\Phi(K))$ is projection. Thus $\core(K/\Phi(K))$ is trivial and accordingly $K$ is trivial (the Frattini subgroup of a nilpotent group is proper). Hence $|G|$ is odd and according to Corollary \ref{superf_cuperc_latin} $Q$ is latin.
\end{proof}

A quandle $Q$ is {\it right-distributive} if 
$$(x*y)*z=(x*z)*(y*z)$$
for every $x,y,z\in Q$, or equivalently, if right translation maps are quandle endomorphisms. Right-distributivity for core quandles can be characterized as follows.

\begin{theorem}[\cite{belousov1963class}]\label{right distr}
    Let $Q=\core(G)$. The following are equivalent
    \begin{itemize}
    \item[(i)] $Q$ is right-distributive.
    \item[(ii)] $G$ is a $2$-Engel group, that is $[x,yxy\mo ]=1$ for every $x,y\in G$.
    \item[(iii)] $x^G$ is a projection subquandle of $\conj(G)$ for every $x\in G$.
    \end{itemize}
\end{theorem}

\begin{proof}
 (i) $\Leftrightarrow$ (ii) The quandle $Q$ is homogeneous, so it is enough to prove $(x*y)*1=(x*1)*(y*1)$ for all $x,y\in G$. We have
 \begin{align*}
(x*y)*1&=xy\mo x xy\mo x\\
(x*1)*(y*1)&=x^2y^{-2} x^2
 \end{align*}
Replacing $y\mo $ by $zx\mo $ we have that right distributivity is equivalent to
 \begin{align*}
z xzx\mo =x zx\mo  zx\mo x=x zx\mo  z.
 \end{align*}

 (ii) $\Leftrightarrow$ (iii) The conjugacy class of $x$ is $x^G=\setof{yxy\mo }{y\in G}$ and it is a subquandle of $\conj(G)$. Thus $x^G$ is a projection subquandle if and only if $(yxy\mo )x(yxy\mo )\mo =x$ for every $y\in G$. 
\end{proof}




We conclude this section by reproving a known theorem about mediality, specifically, an unpublished result by Roszkowska that characterize medial cores. Recall that a quandle $Q$ is {\it medial} if
$$(x*y)*(z*t)=(x*z)*(y*t)$$
for every $x,y,z,t\in Q$. It is known that a quandle is medial if and only if its displacement group is abelian \cite{Medial}.
\begin{theorem}[B. Roszkowska]\label{core medial}
    $\core(G)$ is medial if and only if $G$ is 2-nilpotent.
\end{theorem}

\begin{proof}
If $Q$ is medial then $\dis(Q)$ is abelian and so $G/\zen(G)$ is abelian, according to Corollary \ref{G/Z(G) factor of dis}. On the other hand, if $G$ is $2-$step nilpotent then $[x,y]\in \zen(G)$ and so
\begin{align*}
[\delta_x,\delta_y](z)&=x\mo y\mo xy z y x y\mo  x\mo \\
&=[x,y] z y x y\mo  x\mo \\
&= z y x [x,y]y\mo  x\mo \\
&=z y x x\mo y\mo xy y\mo  x\mo =z
\end{align*}
for every $x,y,z\in G$. Thus $\dis(Q)$ is abelian and then $Q$ is medial. 
\end{proof}

Medial quandles are right distributive. The converse does not hold even in the class of core quandles.

\begin{example}\cite[Example 7.3]{Engel}
Let $p$ be an odd prime and $G$ be the group generated by $g_1, \ldots, g_7$ subject to the following relations:
\begin{align*}
g^p_i &= 1, i = 1,..., 7,\\
[g_1, g_2] &= g_4, [g_1, g_3] = g_5, [g_1, g_6] = g_7^2,\\
[g_2,g_3]&=g_6\mo , [g_2,g_5]=g_7\mo , [g_3,g_4]=g_7, 
\end{align*}
and $[g_i, g_j ]=1$ for all the other pairs $i,j$. The group $G$ has nilpotency class $3$ and it a $2$-Engel group. Thus, $\core(G)$ is right-distributive but not medial.
\end{example}

\section{Commutator theory} \label{sec: commutator theory}

In \cite{CP} we adapted commutator theory in the sense of \cite{comm} to racks and quandles. It turns out that {\it nilpotent} (resp. {\it solvable}) quandles are quandles with nilpotent (resp. solvable) displacement group.

\begin{corollary}
Let $G$ be a group and let $Q=\core(G)$. Then $Q$ is solvable (resp. nilpotent) if and only if $G$ is solvable (resp. nilpotent).
\end{corollary}

\begin{proof}
%
If $\dis(Q)$ is solvable (resp. nilpotent) then $G/\zen(G)$ is solvable (resp. nilpotent) according to Corollary \ref{G/Z(G) factor of dis}. So $G$ is also solvable (resp. nilpotent).

On the other hand $\dis(Q)$ embeds into a factor of $G\times G$ with respect to a central subgroup. So, if $G$ is solvable (resp. nilpotent) then also $\dis(Q)$ is solvable (resp. nilpotent).
\end{proof}

\begin{lemma}\label{semiregularity}
    Let $G$ be a group and let $Q=\core(G)$ be connected. The following are equivalent
    \begin{itemize}
        \item[(i)] $\dis(Q)$ is semiregular.  
        \item[(ii)] If $x_1\cdots x_t =(x_t\cdots x_1)\mo $ then $x_1\cdots x_t\in\zen(G)$.
    \end{itemize}
\end{lemma}

\begin{proof}
The quandle $Q$ is homogeneous, so $\dis(Q)$ is semiregular if and only if $\dis(Q)_1$ (the stabilizers are conjugate). According to Lemma \ref{dis_stab_1},  $\delta_{x_1}\cdots\delta_{x_t}\in \dis(Q)_1$ if and only if $x_1\cdots x_t=(x_t\cdots x_1)\mo $.

    (i) $\implies$ (ii) Assume that $\dis(Q)$ is semiregular, then $\dis(Q)_1=\dis(Q)_a=1$ for all $a\in G$, that is $\delta_{x_1}\cdots\delta_{x_t}\in\dis(Q)_1$ implies that $\delta_{x_1}\cdots\delta_{x_t}=1$. So $\delta_{x_1}\cdots\delta_{x_t}(a)=(x_1\cdots x_t)a(x_t\cdots x_1)=(x_1\cdots x_t)a(x_1\cdots x_t)\mo =a$ for every $a\in G$. Thus, $x_1\cdots x_t\in \zen(G)$.

    (ii) $\implies$ (i): We need to show that $\dis(Q)_1=1$. Assume that $\delta_{x_1}\cdots \delta_{x_t}\in \dis(Q)_1$. Then $x_1\cdots x_t x_t\cdots x_1=1$ and so we have $x_1\cdots x_t\in \zen(G)$. Therefore $x_1\cdots x_t a x_t\cdots x_1 =x_1\cdots x_t a (x_1\cdots x_1)\mo  =a$ for every $a\in G$. Hence $\delta_{x_1}\cdots \delta_{x_t}=1$.
\end{proof}

\begin{proposition}\label{abelian cores}
Let $G$ be a group and $Q=\core(G)$. The following are equivalent
\begin{itemize}
\item[(i)] $Q$ is abelian.
\item[(ii)] $G$ is $2$-step nilpotent and if $x_1\cdots x_t =(x_t\cdots x_1)\mo $ then $x_1\cdots x_t\in\zen(G)$.
\end{itemize}
\end{proposition}

\begin{proof}
A quandle $Q$ is abelian if and only if $\dis(Q)$ is abelian (that is $Q$ is medial) and semiregular \cite[Theorem 1.1]{CP}. Thus, we can put together Lemma \ref{semiregularity} and Theorem \ref{core medial}. 
\end{proof}






Affine quandles are abelian (see \cite[Theorems 2.2,2.3]{jedlivcka2018subquandles}). It is easy to see that if $G$ is abelian, then $\core(G)=\aff(G,-1)$ is affine. The converse does not hold. Indeed, the example in Remark \ref{not iso iff}, shows that $\core(\Z_3^2\rtimes \Z_3)$ is isomorphic to $\core(\Z_3^3)=\aff(\Z_3^3,-1)$.





\section{The Functor Core} \label{sec: functors}

The assignment $G\mapsto \core(G)$ gives rise to a functor from the category of groups to the category of involutory quandles. It is easy to construct the right adjoint of such functor. Let $Q$ be a quandle, we define the {\it $\core$-Adjoint group} of $Q$ by generators and relations as follows:
\begin{align}\label{presentation H(Q)}
\mathcal{H}(Q)=\langle e_x,\, x\in Q\, \colon\, e_{x*y}=e_x e_y\mo  e_x,\, x,y\in Q\rangle.
\end{align}

It is clear that the map 
\begin{align}   \label{map i}
i:Q\to \core(\mathcal{H}(Q)),\quad x\mapsto e_x
\end{align} is a quandle homomorphism.

\begin{proposition}\label{H(Q)}
Let $Q$ be an involutory quandle and $G$ be a group. If $f:Q\to \core(G)$ is a quandle homomorphism then there exists a group homomorphism $\widetilde{f}:\mathcal{H}(Q)\to G$ such that $f=\widetilde{f}\circ i$.
\end{proposition}

\begin{proof}
Assume that $f:Q\to \core(G)$ is a quandle homomorphism. Then we have that $f(x*y)=f(x)f(y)\mo f(x)$. So, by definition of $\mathcal{H}(Q)$ (see \eqref{presentation H(Q)}) there exists a group homomorphism $\widetilde{f}$ mapping $e_x$ to $f(x)$. Hence $\widetilde{f}\circ i =f$. 
\end{proof}

Let $f:Q_1\to Q_2$ a morphism of quandles, then $i_2\circ f$ is a quandle morphisms into $\core(\mathcal{H}(Q_2))=$. Thus, according to Proposition \ref{H(Q)} tha map
$$H(f)=\widetilde{i_2\circ f}:\mathcal{H}(Q_1)\to \mathcal{H}(Q_1),\quad e_x\to e_{f(x)}$$
is a well defined group morphism. So $H(-)$ defines a functor from the category of quandles to the category of groups.

Moreover Proposition \ref{H(Q)} shows that the functor $\core(-)$ is right adjoint to $\mathcal{H}(-)$. Therefore, the pair of functors $(\core(-),\mathcal{H}(-))$ provides an adjunction between the category of groups and the category of involutory quandles (see \cite[Chapter IV]{mclane} for further details).

\begin{proposition}\label{i under lambda}
Let $Q$ be an involutory quandle. Then: 
\begin{enumerate}
\item $\ker{i}\leq \lambda_Q$.
\item The assignment $$\mathcal{H}(Q)\to \lmlt(Q),\quad e_x\to L_x$$
can be extended to a well-defined surjective group homomorphism.
\end{enumerate}
\end{proposition}

\begin{proof}
The map: $f:x\mapsto L_x $ from $Q$ to $\core(\lmlt(Q))$ is a quandle morphism and the kernel is $\lambda_Q$. Therefore there exists a group morphism $\widetilde{f}:\mathcal{H}(Q)\to \lmlt(Q)$ such that $\widetilde{f}\circ i=f$ and so the second statement follows. Moreover, if $i(x)=i(y)$ then $\widetilde{f}(x)=L_x=L_y=\widetilde{f}(y)$, i.e $\ker{i}\leq \lambda_Q$. 
\end{proof}






Following the ideas of \cite{bergman2021core}, we provide several conditions describing when an involutory quandle can be embedded into the core of a group. The first one relies on the injetivity of the map $i$ defined in \eqref{map i}, and the latter is a necessary condition on the 2-generated subquandles of a core.

\begin{corollary}
Let $Q$ be an involutory quandle. Then $Q$ can be embedded into a core quandle if and only if $\ker{i}=0_Q$. 
\end{corollary}

\begin{proof}
If $\ker{i}=0_Q$ then $Q$ embeds into $\core(\mathcal{H}(Q))$. On the other hand if $f:Q\to \core(G)$ is injective, we have that there exists a group morphism $\widetilde{f}:\mathcal{H}(Q)\to G$ with $f=\widetilde{f}\circ i$. Therefore $i$ needs to be injective.
\end{proof}

Note that, if $Q$ is a faithful involutory quandle, then the map $x\mapsto L_x$ is an embedding of $Q$ into $\core(\lmlt(Q))$ (see Proposition \ref{i under lambda}(i)).

\begin{lemma} \label{thm:lemma_LxLy}
    In $Q=\core(G)$, let $x,y\in G$, and let $S=Sg(x,y)$. Then
    \begin{enumerate}
        \item $(L_xL_y)^i(x)=x(y\mo x)^{2i}$, for all $i\in\N$.
        \item $(L_xL_y)^i(y)=y(y\mo x)^{2i}$, for all $i\in\N$.
        \item $|{L_xL_y|}_{S}|=|(y\mo x)^2|=|x^{L_x L_y}|=|y^{L_x L_y}|$.
    \end{enumerate}
\end{lemma}

\begin{proof}
    Items (i) and (ii) follow by a direct computation. The order of the automorphism $L_x L_y|_S$ equal the least common multiple of the lengths of the orbits of the generators of $S$. So (iii) follows from the statements (i) and (ii). 
\end{proof}

\begin{lemma}
    Let $Q=\core(G)$ and let $x,y\in G$. Then $Sg(x,y)$ is affine (hence abelian).
\end{lemma}

\begin{proof}
    Let $S=Sg(x,y)$. According to \cite[Lemma 3.1]{Super}, $\dis(S)=\langle L_x L_y\rangle$ and $S=x^{\dis(S)}\cup y^{\dis(S)}$ by virtue of \cite[Corollary 5.5]{semimedial}. Then according to Lemma \ref{thm:lemma_LxLy} $S=\setof{x(y\mo x)^{2i}, \, y(y\mo x)^{2i}}{i\in \mathbb{N}}$. 
    
    If $S$ is connected, then $S$ is affine. Assume $S$ not connected, let $C=\Z_{2n}$ if $n=|{L_xL_y|}_{S}|$ or $C=\Z$ if the order of $L_x L_y$ is infinite. The map:
    \[
    \begin{aligned}
        \varphi\colon &S \to \aff(C,-1) \\
        & x(y\mo x)^{2i} \mapsto 2i \\
        & y(y\mo x)^{2j} \mapsto 1+2j
    \end{aligned}
    \]
    is a quandle isomorphism, hence $S$ is affine. 
\end{proof}

\begin{corollary} \label{cor:2_gen_are_affine}
Let $Q$ be an involutory quandle. If $Q$ embeds into a core quandle, then $Sg(x,y)$ is affine for every $x,y\in Q$.
\end{corollary}

The converse of Corollary \ref{cor:2_gen_are_affine} does not hold. According to \cite[Proposition 3.1]{Super}, $2$-generated connected involutory quandles are affine. Let $Q$ be the involutory quandle \texttt{SmallQuandle(15,6)} of the RIG database. The quandle $Q$ is latin and so its $2$-generated subquandles are affine. On the other hand, $Q$ is not isomorphic to the core of any of the groups of order 15. 


\section{Coloring by core quandles} \label{sec: coloring} 
Let $\mathcal{K}$ be a knot and $Q$ a quandle. A \emph{quandle coloring} of $\mathcal{K}$ by $Q$ is an assignment of elements of $Q$ to the arcs of $\mathcal{K}$ such that the \emph{crossing relation} as shown in Figure \ref{fig:crossing} holds in $Q$ (see \cite{fish2015combinatorial} and \cite{fish2016efficient} for a complete presentation of the theory of knots and colorings).

\begin{figure}[ht]
    \centering
    \includegraphics[width=.5\linewidth]{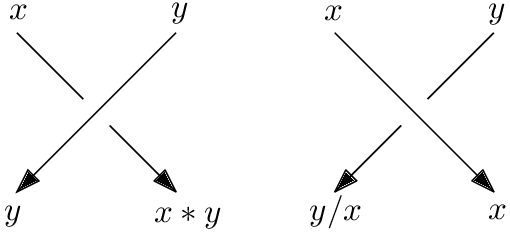}
    \caption{Crossing relation.}
    \label{fig:crossing}
\end{figure}

These crossing relations ensure that the number of colorings by a given quandle is invariant under the Reidemeister moves, making it a topological invariant of the knot. A \emph{trivial coloring} of a knot diagram using a quandle $Q$ is a quandle coloring in which the same element of $Q$ is assigned to all arcs of the diagram. We say that a knot is {\it colorable} by a quandle $Q$ if the knot admits a non-trivial coloring by $Q$.

Given a diagram of a knot $\mathcal{K}$, its knot quandle $\mathcal{Q}(\mathcal{K})$ is the free quandle on the arc-set of $\mathcal{K}$ modulo the equivalence relations generated by the crossing relations. A coloring $c$ of a knot $\mathcal{K}$ using a quandle $Q$ can be seen as a quandle homomorphism $c:\mathcal{Q}(\mathcal{K})\to Q$.

Let $Q$ be a quandle. Following \cite[Section 2.2]{Orbits} we define
\begin{align*}
    L_0(Q),\qquad L_{n+1}(Q)=L_n(Q)/\lambda_{L_n(Q)}, \quad \text{ for every } n\in \mathbb{N}.
\end{align*}
    
\begin{lemma}\label{L_n}
    Let $Q$ be a finite connected quandle. Then there exists $n\in \mathbb{N}$ such that $L_n(Q)$ is faithful and $|L_n(Q)|>1$.
\end{lemma}

\begin{proof}
    Assume that $m$ is the smallest integer such that $|L_m(Q)|=1$. Then $L_{m-1}(Q)$ is connected and projection factor of $Q$, so $|L_{m-1}(Q)|=1$, contradiction. Therefore $|L_n(Q)|>1$ for every $n\in \mathbb{N}$. Since $|L_{k+1}(Q)|\leq |L_k(Q)|$ for every $k\in \mathbb{N}$, there exists $n\in \mathbb{N}$ such that $|L_n(Q)|=|L_{n+1}(Q)|$, that is $L_n(Q)$ is faithful.
\end{proof} 

\begin{lemma}\label{coloring by Cayley}
Let $\mathcal{K}$ be a knot and $Q$ be a quandle. If $\mathcal{K}$ is colorable by $Q$ then $\mathcal{K}$ is colorable by $L_n(Q)$ for every $n\in \mathbb{N}$.
\end{lemma}

\begin{proof}
It is enough to prove the statement for $n=1$ (because $L_{n+1}(Q)=L_1 (L_n(Q))$).

Let $Q_{\mathcal{K}}$ be the quandle knot of $\mathcal{K}$. Assume that $c:Q_{\mathcal{K}} \to Q$ is a coloring. The map $$c':Q_{\mathcal{K}}\to Q\to Q/\lambda_Q,\quad x\mapsto [c(x)]_{\lambda_Q}$$ is a quandle coloring. Assume that the image is trivial, that is $x\,\lambda_Q\,y$ for every $x,y\in \im(c)$. Then $\im(c)$ is a projection and connected quandle. Hence $|\!\im(c)|=1$, contradiction. 
\end{proof}

\begin{corollary}
Let $\mathcal{K}$ be a knot. If $\mathcal{K}$ can be colored by a finite quandle, then $\mathcal{K}$ can be colored by a faithful quandle.
\end{corollary}

\begin{proof}
Put together Lemma \ref{L_n} and Lemma \ref{coloring by Cayley}.
%
\end{proof}

\begin{corollary}
Let $\mathcal{K}$ be a knot. If $\mathcal{K}$ can be colored by a finite involutory quandle, then $\mathcal{K}$ can be colored by a core quandle.
\end{corollary}

\begin{proof}
Let $Q$ be an involutory quandle and $c:Q_{\mathcal{K}}\to Q$ be a coloring. Then we have a non-trivial coloring by a faithful quandle $Q'$ that embeds into $\core(\lmlt(Q'))$. Hence we have a coloring of $\mathcal{K}$ by $\core(\lmlt(Q'))$.
\end{proof}

\begin{example}
    Let $\mathcal{K}$ be a torus knot. In \cite[Section 4]{spaggiari2023coloring}, we show that $\mathcal{K}$ can there exists a coloring $c$ by $\conj(\D_n)$, the conjugation quandle over the dihedral group. Necessarily, $\im(c)$ is a subquandle of an isomorphic copy of $\core(\Z_n)$ contained in $\conj(\D_n)$.
\end{example}

    

\printbibliography

\newpage

\end{document}